\documentclass[11pt,a4paper,reqno]{amsproc}
\usepackage{cite}
\usepackage[pdftex,bookmarks=true]{hyperref}
\newtheorem{theorem}{Theorem}[section]

\usepackage{mathrsfs}
\theoremstyle{definition}
\newtheorem{definition}[theorem]{Definition}

\theoremstyle{remark}

\numberwithin{equation}{section}
\usepackage{graphicx}
\graphicspath{ {images/} }
\newcommand{\xdownarrow}[1]{%
  {\left\downarrow\vbox to #1{}\right.\kern-\nulldelimiterspace}
}
\newcommand{\xuparrow}[1]{%
  {\left\uparrow\vbox to #1{}\right.\kern-\nulldelimiterspace}
}

\begin{document}
\title[ Mehler-Fock transform]{Reverse Translation and Convolution Estimates in the Framework of the Mehler-Fock Transform}

\author[U. K. Mandal, A. Prasad and S. Ranjan ] {U. K. Mandal$^{\dagger}$, Akhilesh Prasad$^{*}$ and  Sudhanshu Ranjan$^{\dagger \dagger }$ }


\date{}


\maketitle

$^\dagger$ Department of Mathematics, Nalanda College, Biharsharif, Patliputra University, Patna, India\\
$^{ *, \dagger \dagger}$ Department of  Mathematics and Computing, Indian Institute of Technology (Indian School of Mines), Dhanbad, India

E-mail: {{upainmandal$@$gmail.com$^{\dagger}$ (U. K. Mandal) ;~ apr$\_$bhu$@$yahoo.com$^{*}$ (Corresponding Author: Akhilesh Prasad ); ~ sudhanshum$1997@$gmail.com$^{ \dagger \dagger}$ (Sudhanshu Ranjan)}

\begin{abstract}
In this paper, reverse translation and  convolution estimates in the frame of $\mu^{th}$ order Mehler-Fock transform are obtained. Continuity properties of these operators on the test function space $\mathfrak{H}_{l,r,n}$ have been discussed. Moreover, Sobolev space $S_{r,p}^{N,\alpha}$ has been introduced and again continuity properties of these operators in this Sobolev space have been discussed.
\end{abstract}
\section{Introduction}
The zero order Mehler-Fock transform (MFT) is an integral transform that is generated by  a special function namely Legendre function $P_{-\frac{1}{2}+i\eta}(u),$ where $ ~\eta\in(0,\infty),~u\in (1,\infty)$. The MFT emerged as the spectral transform involving a Sturm-Liouville problem whose eigenfunctions are Legendre  functions. It was first introduced by Mehler \cite{mehler}, and later Fock \cite{fock}, studied about such transformation in a systematic way. Following that, several studies have been carried out in both applied and theoretical areas concerning the continuity properties of such transforms and their associated operators in different function spaces; for instance, see \cite{PS, Bn1, Yaku2,Sai,TP}.

The investigation associated to MFT extended by invoking  the associated Legendre function $P^{-\mu}_{-\frac{1}{2}+i\eta}(u),$ as its kernel, where the order $\mu$  may be real or complex, later certain conditions were deployed over the order as per restrictions over the transformation. This extension  opened a space of investigation in the wider class of harmonic analysis. Various theories have been developed and corresponding applications have been done in this context, see \cite{gen1,gen3,gen4,gen5,gen7,gen8,gonz1, Ha1}.

The $\mu^{th}$ order Mehler-Fock transform ($\mu$MFT) of any integrable function $\varphi$  is defined \cite{ gonz1,Ha1}as follows:
\begin{eqnarray}
	\big(\mathfrak{MF_\mu}\varphi\big)(\eta)=\int_1^\infty \varphi(u) P_{-\frac{1}{2}+i\eta}^{-\mu}(u)  du,\quad \eta\in \mathbb{R}_+=(0,\infty), \label{1.1}
\end{eqnarray}  
The inversion of $\mu$MFT defined in \eqref{1.1} is given by 
 \begin{eqnarray}
	\varphi(u)= \frac{1}{\pi}\int_0^\infty\eta \sinh(\pi\eta)\Phi(\eta,\mu)P_{-\frac{1}{2}+i\eta}^{-\mu}(u)\big(\mathfrak{MF_\mu}\varphi\big)(\eta)d\eta, \label{1.2}
\end{eqnarray}
with $\Phi(\eta,\mu)=\Gamma(\mu+\frac{1}{2}+i\eta)\Gamma(\mu+\frac{1}{2}-i\eta)$ and $u\in \mathbb{I}=(1,\infty)$, where $P_{\frac{-1}{2}+i\eta}^{-\mu}$ is of order $\mu\in \mathbb{C}$ and $~\mathfrak{R}\mu>\frac{-1}{2}$.
The expression for $P_{-\frac{1}{2}+i\eta}^{-\mu}(u)$ is formulated as 
\begin{eqnarray*}
	P_{-\frac{1}{2}+i\eta}^{-\mu}(u) =\frac{1}{2^{\mu}\Gamma(\mu+1)}(u^2-1)^\frac{\mu}{2} {_{2}F_1}\bigg(\mu+\frac{1}{2}+i\eta, \mu+\frac{1}{2}-i\eta ;\mu+1; \frac{1-u}{2}\bigg).
\end{eqnarray*} 
From \cite{gonz1}, we have
\begin{eqnarray}
	|P_{-\frac{1}{2}+i\eta}^{-\mu}(u)|\leq \frac{\Gamma(\frac{1}{2}+\mathfrak{R \mu})}{|\Gamma(\frac{1}{2}+\mu)|}P^{-\mathfrak{R}\mu}_{-\frac{1}{2}}(u),\label{1.3}
\end{eqnarray}
\begin{eqnarray}
	P^{-\mathfrak{R}\mu}_{-\frac{1}{2}}(u)\sim \frac{(u-1)^{\frac{\mu}{2}}}{2^{\frac{\mu}{2}}\Gamma(\mu+1)}, \quad u\rightarrow 1^+, \label{1.4}
\end{eqnarray}
\begin{eqnarray}
	P^{-\mathfrak{R}\mu}_{-\frac{1}{2}}(u)\sim \frac{1}{\Gamma(\mu+\frac{1}{2})}\Big(\frac{2}{\pi u}\Big)^{\frac{1}{2}}\ln u,\quad u\rightarrow \infty.\label{1.5}
\end{eqnarray}
According to \cite{Ha1,Glaeske}, the differential operator  $A_u$ is given by 
\begin{eqnarray}
	A_u&=&(u^2-1)^\frac{-\mu}{2} D_u(u^2-1)^{\mu+1}D_u(u^2-1)^\frac{-\mu}{2},\quad D_u=\frac{d}{du},\label{1.6}
	\\\nonumber
	&=&(u^2-1)\frac{d^2}{du^2}+2u\frac{d}{du}- \biggl\{\mu(\mu+1)+\frac{\mu^2}{u^2-1}\biggr\},
\end{eqnarray}
For $k\in \mathbb{N}_0$, from \cite{gonz1, Glaeske}, we have
\begin{eqnarray}
	A_u^kP_{-\frac{1}{2}+i\eta}^{-\mu}(u)=(-1)^k\biggl[\Big(\mu+\frac{1}{2}\Big)^2+\eta^2 \biggr]^k P_{-\frac{1}{2}+i\eta}^{-\mu}(u).\label{1.7}
\end{eqnarray}
 The series representation of $A_u$ for $k\in \mathbb{N}_0$ is given by \cite{Ha1,Glaeske}
 \begin{eqnarray}
 	A^k_uf(u)=\sum^{2k}_{j=0}(u^2-1)^{j-k} p_{j,k}(u)D_u^j f(u),\label{1.8}
 \end{eqnarray}
 where polynomials $p_{j,k}(u)$ are  of degree $2k-j$, specifically  $p_{2k,k}(u)=1$ and $p_{2k-1,k}(u)=2k^2u$. Another form can also be defined as follows:
\begin{eqnarray}
	A^k_uf(u)=(u^2-1)^\frac{\mu}{2}\sum^{2k}_{j=1}q_{j,k}(u)D_u^j[(u^2-1)^\frac{-\mu}{2}f(u)],\label{1.9}
\end{eqnarray}
where each $q_{j,k}(u)$ are polynomials of degree $j$ and does not depend on $f$.\\
The associated Legendre function $P_{-\frac{1}{2}+i\eta}^{-\mu}(\cdot)$ has the following product formula \cite[p. 259]{Ha1} 
\begin{eqnarray*}
	P_{-\frac{1}{2}+i\eta}^{-\mu}(u)P_{-\frac{1}{2}+i\eta}^{-\mu}(v)=\int_1^\infty K(u,v,w)P_{-\frac{1}{2}+i\eta}^{-\mu}(w) dw,\label{1.10}
\end{eqnarray*}
which on taking inverse gives the integral representation of the symmetric kernel $K(u,v,w)$ as 
\begin{eqnarray}
K(u,v,w)=  \frac{1}{\pi}\int_0^\infty\eta \sinh(\pi\eta)\Phi(\eta,\mu)P_{-\frac{1}{2}+i\eta}^{-\mu}(u)P_{-\frac{1}{2}+i\eta}^{-\mu}(v)P_{-\frac{1}{2}+i\eta}^{-\mu}(w)d\eta, \quad\label{1.10}
\end{eqnarray}
with the symmetric kernel 
\begin{eqnarray}
\nonumber
	K(u,v,w) = \begin{cases}
		\frac{2^{\mu}\Gamma(\mu+1)}{\pi\Gamma(2\mu+1)}Y(u,v,w)^{\mu-\frac{1}{2}}\\
        ~~~\times[(u^2-1)(v^2-1)(w^2-1)]^{-\frac{\mu}{2}}, &\quad \text{for}\quad Y>0
		, \\ 
		0 &\quad \text{for}\quad Y\leq0,
	\end{cases}	
	\\
	\label{1.11}
\end{eqnarray}
where $Y(u,v,w)=2uvw+1-u^2-v^2-w^2$.
\\
Followed to this result the generalized translation operator $_\mu\mathfrak{T}_u$ for $f\in L^1(\mathbb{I};~ du)$ is given by \cite{Glaeske, Ha1}
\begin{eqnarray}
	(_\mu\mathfrak{T}_u f)(v)=(_\mu\mathfrak{T}_v f)(u)=\int^\infty_1 K(u,v,w)f(w)dw.\label{1.12}
\end{eqnarray}
For $f, g \in L^1(\mathbb{I}; du)$, the convolution operator is given as
\begin{eqnarray}
	(f\underset{\mu}{*}g)(u)&=&\int^\infty_1 g(v)(_\mu\mathfrak{T}_u f)(v)dv\nonumber
	\\
	&=&\int^\infty_1\int^\infty_1 f(w)g(v) K(u,v,w)dvdw.\label{1.13}
\end{eqnarray}

The paper is organized as follows. Section 1 provides an introduction to the $\mu^{th}$ order Mehler-Fock transform and presents the necessary preliminary results, including the definitions, asymptotic relations, and other basic results on which the paper is based. Section 2 deals with reverse translation and convolution estimates obtained by means of the reverse H\"older's inequality. Section 3 is devoted to the study of continuity properties of these operators, along with the adjoint $\mu$MFT. Finally, Section 4 is concerned with the continuity properties of convolution and translation operators on the Sobolev type space.

\section{The reverse translation and convolution estimates}
Forward convolution estimates in the framework of various integral transforms have been developed in classical harmonic analysis using Schwartz's and H\"older's inequalities to study the boundedness properties of the associated operators. In contrast, reverse estimates are of particular importance in the study of inverse problems and uniqueness questions for partial differential equations, etc.; see \cite{reverse1,reverse5,application1,application2,application3,application4}.  To obtain the reverse convolution estimate, the reverse H\"older's inequality plays a key role \cite{holder}. However, such estimates for the $\mu^{\text{th}}$ order Mehler–Fock transform are still lacking. This gap provides the motivation for the present section.\\ 
Using \eqref{1.11}, first we obtain a lower bound of the function $K(u,v,w)$.
\\
For $u,~v,~w \in \mathbb{I}$, we can write
\begin{eqnarray*}
(u^2-1)(v^2-1)(w^2-1)\leq u^2v^2w^2.
\end{eqnarray*}
Which implies that, for $\mathfrak{R}\mu>\frac{-1}{2}$ 
\begin{eqnarray}
|(u^2-1)(v^2-1)(w^2-1)|^{\frac{-\mathfrak{R}\mu}{2}}\geq (uvw)^{-\mathfrak{R}\mu}.\label{2.1}
\end{eqnarray}
Now in view of \eqref{1.11} and \eqref{2.1}, we can write that the lower bound for $K(u,v,w)$ is
\begin{eqnarray}
|K(u,v,w)|\geq C_{\mu,\epsilon} (uvw)^{-\mathfrak{R}\mu},\label{2.2}
\end{eqnarray}
where $\epsilon>0$ be any fix number for which $Y\geq \epsilon>0$ and $C_{\mu,\epsilon} >0$ be any constant. 
\\
To obtain the reverse convolution inequality, we will use the above estimate as well as the reverse H\"older's inequality stated as
 \cite{holder}:
\begin{theorem}
If $F_1>0$ and $F_2>0$ two functions which satisfies the condition $0<c\leq \frac{F_1}{F_2}\leq  C<\infty$ on the non-empty set X, then
\begin{eqnarray}
\left(\int\limits_X F_1~ d\sigma  \right)^{\frac{1}{p}}\left(\int\limits_X F_2~ d\sigma  \right)^{\frac{1}{q}}\leq M_{p,q} \left(\frac{c}{C}\right) \int\limits_X F_1^\frac{1}{p}F_2^\frac{1}{q}~ d\sigma,\label{2.3}
\end{eqnarray}
if involved integrals exist, where $M_{p,q}(s)=\frac{s^{-\frac{1}{pq}}(1-s)}{p^{\frac{1}{p}}(1-s^{\frac{1}{p}})^\frac{1}{p} q^{\frac{1}{q}}(1-s^{\frac{1}{q}})^{\frac{1}{q}}}$,~$p, q>1$ and $\frac{1}{p}+\frac{1}{q}=1$.
\end{theorem}

 \begin{theorem}
 Let $f\in L^p(\mathbb{I};~w^{-\mathfrak{R}\mu}dw)$ be a function which satisfies the condition $0<c^\frac{1}{p}\leq f\leq C^{\frac{1}{p}}<\infty$, where $c$ and $C$ are constants  then for $1<p<\infty$, we have
 \begin{eqnarray}
 \|_\mu\mathfrak{T}_uf\|_{L^1(\mathbb{I};~v^{-\mathfrak{R}\mu}dv)} \geq C^{\prime\prime}_{\mu,\epsilon}\Big[M_{p,q}\Big(\frac{c}{C}\Big)\Big]^{-1}\|f\|_{L^p(\mathbb{I};~w^{-\mathfrak{R}\mu}dw)},\label{2.4}
\end{eqnarray}  
where $C^{\prime\prime}_{\mu,\epsilon}>0$ is constant and $\mathfrak{R}\mu>1$. 
 \end{theorem}
 \begin{proof}
 Let $F_1(w)=f^p(w)K(u,v,w)>0$ and $F_2(w)=K(u,v,w)>0$ are two functions satisfying the condition $0<c\leq \frac{F_1}{F_2}\leq  C<\infty$. Then using \eqref{2.3} and then \eqref{2.2}, we have
 \begin{eqnarray*}
 M_{p,q}\Big(\frac{c}{C}\Big)(_\mu\mathfrak{T}_uf)(v)&=& M_{p,q}\Big(\frac{c}{C}\Big)\int\limits_1^\infty F_1^{\frac{1}{p}}(w)F_2^{\frac{1}{q}}(w)dw
 \\
 &\geq& \left(\int\limits_1^\infty F_1(w)dw \right)^{\frac{1}{p}}\left(\int\limits_1^\infty F_2(w)dw \right)^{\frac{1}{q} }
 \\
  &\geq& C_{\mu,\epsilon} (uv)^{-\mathfrak{R}\mu}\|f\|_{L^p(\mathbb{I};~w^{-\mathfrak{R}\mu}dw)} \left(\int\limits_1^\infty w^{-\mathfrak{R}\mu} dw \right)^{\frac{1}{q} }.
 \end{eqnarray*}
 Involved integral is convergent for $\mathfrak{R}\mu>1$. Thus
 \begin{eqnarray*}
 (_\mu\mathfrak{T}_uf)(v)&\geq& C^\prime_{\mu,\epsilon}\Big[M_{p,q}\Big(\frac{c}{C}\Big)\Big]^{-1} (uv)^{-\mathfrak{R}\mu}\|f\|_{L^p(\mathbb{I};~w^{-\mathfrak{R}\mu}dw)}. 
 \end{eqnarray*}
Thus
\begin{eqnarray*}
\|_\mu\mathfrak{T}_uf\|_{L^p(\mathbb{I};~v^{-\mathfrak{R}\mu}dv)} &\geq& C^{\prime\prime}_{\mu,\epsilon}\Big[M_{p,q}\Big(\frac{c}{C}\Big)\Big]^{-1} u^{-\mathfrak{R}\mu}\|f\|_{L^p(\mathbb{I};~w^{-\mathfrak{R}\mu}dw)},
\end{eqnarray*}
where $C^{\prime\prime}_{\mu,\epsilon}>0$ is constant. Further, for some  fixed $u\in \mathbb{I}, ~u^{-\mathfrak{R}\mu}>1$. So , we can rewrite the above estimate as
\begin{eqnarray*}
\|_\mu\mathfrak{T}_uf\|_{L^p(\mathbb{I};~v^{-\mathfrak{R}\mu}dv)} &\geq& C^{\prime\prime}_{\mu,\epsilon}\Big[M_{p,q}\Big(\frac{c}{C}\Big)\Big]^{-1} \|f\|_{L^p(\mathbb{I};~w^{-\mathfrak{R}\mu}dw)}.
\end{eqnarray*}
Hence, \eqref{2.4} is obtained.
 \end{proof}
 \begin{theorem}
 Let $f\in L^p(\mathbb{I};~w^{-\mathfrak{R}\mu}dw), ~g\in L^p(\mathbb{I};~v^{-\mathfrak{R}\mu}dv)$ are two function satisfying the conditions 
\begin{eqnarray}
 0<f\leq C_1<\infty\quad \& \quad 0<g\leq C_2<\infty, \label{2.5}
\end{eqnarray} 
where $C_1,C_2>0$ are constants and $1<p<\infty$, then reverse convolution estimate is given as
 \begin{eqnarray}
 \|f\underset{\mu}{*}g\|_{L^p(\mathbb{I};~u^{-\mathfrak{R}\mu}du)}\geq C^\prime_{\mu,\epsilon,p}   \|f\|_{L^p(\mathbb{I};~w^{-\mathfrak{R}\mu}dw)}\|g\|_{L^p(\mathbb{I};~v^{-\mathfrak{R}\mu}dv)},\label{2.6}
 \end{eqnarray}
 where $C^\prime_{\mu,\epsilon,p}>0 $ is constant and $\mathfrak{R}\mu>\frac{1}{2}$.
 \end{theorem}
 \begin{proof}
Using \eqref{1.13}, we have
\begin{eqnarray}
(f^p\underset{\mu}{*}g^p)(u)=\int\limits_1^\infty\int\limits_1^\infty K(u,v,w)f^p(w)g^p(v)dvdw.\label{2.7}
\end{eqnarray}
In view of \eqref{2.5} and \eqref{2.7}, we can write
\begin{eqnarray}
(f^p\underset{\mu}{*}g^p)(u) \leq (C_1C_2)^{p-1}(f\underset{\mu}{*}g)(u).\label{2.8}
\end{eqnarray}
Now, from \eqref{2.7}, \eqref{2.8} and \eqref{2.2}, we get
\begin{eqnarray*}
(f\underset{\mu}{*}g)(u)\geq C_{\mu,\epsilon} (C_1C_2)^{1-p} u^{-\mathfrak{R}\mu} \|f\|_{L^p(\mathbb{I};~w^{-\mathfrak{R}\mu}dw)}\|g\|_{L^p(\mathbb{I};~v^{-\mathfrak{R}\mu}dv)}.
\end{eqnarray*}
Thus
\begin{eqnarray*}
\|f\underset{\mu}{*}g\|_{L^p(\mathbb{I};~u^{-\mathfrak{R}\mu}du)}\geq C^\prime_{\mu,\epsilon,p}   \|f\|_{L^p(\mathbb{I};~w^{-\mathfrak{R}\mu}dw)}\|g\|_{L^p(\mathbb{I};~v^{-\mathfrak{R}\mu}dv)},
\end{eqnarray*}
where $C^\prime_{\mu,\epsilon,p}>0 $ is constant and $\mathfrak{R}\mu>\frac{1}{2}$. Hence, \eqref{2.6} is obtained.
 \end{proof}
\begin{theorem}
 Let $\rho_1$ and $\rho_2$ be two positive functions. Suppose that $f\in L^p(\mathbb{I};~w^{-\mathfrak{R}\mu}\rho_1(w) dw)$ and $ ~g\in L^p(\mathbb{I};~v^{-\mathfrak{R}\mu} \rho_2(v) dv)$ satisfy the conditions
\begin{eqnarray}
 0<c_1^{\frac{1}{p}}\leq f\leq C_1^{\frac{1}{p}}<\infty, \quad \text{and} \quad 0<c_2^{\frac{1}{p}}\leq g\leq C_2^{\frac{1}{p}}<\infty, \label{2.9}
\end{eqnarray} 
where $c_1,~c_2,~C_1,~C_2>0$ are constants and $1<p<\infty$. Then reverse convolution estimate is obtained as
\begin{eqnarray}
\nonumber
&&\|(f\rho_1\underset{\mu}{*}g\rho_2)\big(\rho_1\underset{\mu}{*}\rho_2\big)^{\frac{1}{p}-1}\|_{L^p(\mathbb{I};~u^{-\mathfrak{R}\mu}du)}
\\
 &\geq &C^\prime_{\mu,\epsilon,p} \left [M_{p,q}\Big(\frac{c_1c_2}{C_1C_2}\Big)\right]^{-1}  \|f\|_{L^p(\mathbb{I};~w^{-\mathfrak{R}\mu}dw)}
 \|g\|_{L^p(\mathbb{I};~v^{-\mathfrak{R}\mu}dv)},\label{2.10}
\end{eqnarray}
where $C^\prime_{\mu,\epsilon,p}>0$ is constant and $\mathfrak{R}\mu>\frac{1}{2}$.
\end{theorem} 
\begin{proof}
We assume that  $F_1(v,w)=f^{p}(w)g^{p}(v)\rho_1(w)\rho_2(v)K(u,v,w)>0$ and 
$F_2(v,w)=\rho_1(w)\rho_2(v)K(u,v,w)>0$, such that $0<c\leq \frac{F_1}{F_2}\leq C<\infty$. Then, using \eqref{1.13}, and reverse H\"older's inequality \eqref{2.3}, we have
\begin{eqnarray*}
&& M_{p,q}\Big(\frac{c_1c_2}{C_1C_2}\Big)(f\rho_1\underset{\mu}{*}g\rho_2)(u) \\
&&=M_{p,q}\Big(\frac{c_1c_2}{C_1C_2}\Big) \int\limits_0^\infty\int\limits_0^\infty F^{\frac{1}{p}}_1(v,w)F^{\frac{1}{q}}_2(v,w) dvdw\\
&&\geq  \Big(\int\limits_0^\infty\int\limits_0^\infty F_1(v,w) dvdw \Big)^{\frac{1}{p}}\Big(\int\limits_0^\infty\int\limits_0^\infty F_2(v,w) dvdw \Big)^{\frac{1}{q}}
\\
&&= \Big(\int\limits_0^\infty\int\limits_0^\infty f^{p}(w)g^{p}(v)\rho_1(w)\rho_2(v)K(u,v,w) dvdw \Big)^{\frac{1}{p}}\big((\rho_1\underset{\mu}{*}\rho_2)(u)\big)^{\frac{1}{q}}.
\end{eqnarray*}
Using \eqref{2.2}, we can write
\begin{eqnarray*}
&&\big((f\rho_1\underset{\mu}{*}g\rho_2)(u) \big)^p\big((\rho_1\underset{\mu}{*}\rho_2)(u)\big)^{1-p}
\\
&&\geq C_{\mu,\epsilon} \left [M_{p,q}\Big(\frac{c_1c_2}{C_1C_2}\Big)\right]^{-p} u^{-\mathfrak{R}\mu} \|f\|^p_{L^p(\mathbb{I};~w^{-\mathfrak{R}\mu}dw)}\|g\|^p_{L^p(\mathbb{I};~v^{-\mathfrak{R}\mu}dv)}.
\end{eqnarray*}
Thus
\begin{eqnarray*}
&&\|(f\rho_1\underset{\mu}{*}g\rho_2)\big(\rho_1\underset{\mu}{*}\rho_2\big)^{\frac{1}{p}-1}\|_{L^p(\mathbb{I};~u^{-\mathfrak{R}\mu}du)}
\\
 &&\geq C^\prime_{\mu,\epsilon,p} \left [M_{p,q}\Big(\frac{c_1c_2}{C_1C_2}\Big)\right]^{-1}  \|f\|_{L^p(\mathbb{I};~w^{-\mathfrak{R}\mu}dw)}
 \|g\|_{L^p(\mathbb{I};~v^{-\mathfrak{R}\mu}dv)},
\end{eqnarray*}
where $C^\prime_{\mu,\epsilon,p}>0$ is constant and $\mathfrak{R}\mu>\frac{1}{2}$. Thus, \eqref{2.10} is obtained.
\end{proof}

 \section{Continuity of translation and convolution operators}
 The literature contains several studies on the continuity properties of integral transforms, together with their associated convolution and translation operators on certain test function spaces; see \cite{index,Lebedev3,Hm,Sn,Gla1,Yaku,AP,AP2}. Inspired by these works, we introduce in this section a test function space and prove continuity theorems for the $\mu$MFT.

 \par 
 We consider  the test function space $\mathfrak{H}_{l,r,n}$, where $l,r\geq 0,~n\in \mathbb{N}_0$, analogous to the function space as defined in \cite[p. 169]{AP} containing  functions $\varphi$,  such that 
 \begin{eqnarray}
 \lambda_{l,r,n}(\varphi)=\sup\limits_{u\in \mathbb{I}}|e^{-lu}(u^2-1)^r A^n_u \varphi(u)|<\infty.\label{3.1}
 \end{eqnarray}
\begin{enumerate}
\item As an example, consider an example $\varphi(u)=e^{-u}$, which belongs to the above-defined function space $\mathfrak{H}_{l,r,1}, ~l\geq 0,~r\geq 1$.
\item As another example, consider another example $\varphi(u)=(u-1)^n e^{-u}$, which belongs to the above-defined function space $\mathfrak{H}_{l,r,n},$ where $ ~l,r\geq 0,$ $n\in \mathbb{N}_0$.
\item For a fixed $\eta>0$, the Legendre function $P^{-\mu}_{-\frac{1}{2}+i\eta}(u)$ belongs to the function space $\mathfrak{H}_{l,r,n}$ for $ l\geq 0,$$~r\geq \frac{1}{4}$ and $ n\in \mathbb{N}_0$.
\end{enumerate}
\begin{theorem}
If the translation operator  $_\mu\mathfrak{T}_u$ is defined as \eqref{1.12}, then it is a continuous linear mapping from  function space $\mathfrak{H}_{0,m,n}, $ $ \frac{1}{4}<m<\frac{3}{4},$ ~\\$ ~n\in\mathbb{N}_0$ into $\mathfrak{H}_{l,r,n},~l>0,~r\geq\frac{1}{4}, ~n\in \mathbb{N}_0$.
\end{theorem} 
\begin{proof}
Using \eqref{1.12}, we have
\begin{eqnarray*}
A_u^n (_\mu\mathfrak{T}_u\varphi)(v)= \int\limits_1^\infty A^n_uK(u,v,w) \varphi(w)dw
= \int\limits_1^\infty A^n_wK(u,v,w) \varphi(w)dw.
\end{eqnarray*}
 The differential operator $A_u$ is self adjoint that is $<A_uf,g>=<f,A_ug>$ for some function $f$ and $g$ defined over suitable function space, such that
\begin{eqnarray*}
\langle f(w),g(w)\rangle = \int\limits_1^\infty f(w)g(w)dw.
\end{eqnarray*}
Thus, In view of \eqref{1.7}, \eqref{1.10} and the self adjoint property of $A_u$ as described above, we may write
\begin{eqnarray}
\nonumber
A_u^n (_\mu\mathfrak{T}_u\varphi)(v)&=& \int\limits_1^\infty \Bigg(\frac{1}{\pi}(-1)^n \int\limits_0^\infty \Big[\Big(\mu+\frac{1}{2} \Big)^2+\eta^2 \Big]^n \eta \sinh(\pi\eta)\Phi(\eta,\mu) 
\\\nonumber
&\times&P^{-\mu}_{\frac{-1}{2}+i\eta}(u)P^{-\mu}_{\frac{-1}{2}+i\eta}(v)P^{-\mu}_{\frac{-1}{2}+i\eta}(w)d\eta\Bigg)\varphi(w)dw
\\\nonumber
&=& \int\limits_1^\infty \Bigg(\frac{1}{\pi} \int\limits_0^\infty  \eta \sinh(\pi\eta)\Phi(\eta,\mu) P^{-\mu}_{\frac{-1}{2}+i\eta}(u)P^{-\mu}_{\frac{-1}{2}+i\eta}(v)
\\\nonumber
&\times&\big(A_w^nP^{-\mu}_{\frac{-1}{2}+i\eta}(w)\big) d\eta\Bigg)\varphi(w)dw
\\\nonumber
&=& \frac{1}{\pi}\int\limits_0^\infty \Bigg(\int\limits_1^\infty \big(A_w^nP^{-\mu}_{\frac{-1}{2}+i\eta}(w)\big) \varphi(w)dw\Bigg) \eta \sinh(\pi\eta)\Phi(\eta,\mu) 
\\\nonumber
&\times& P^{-\mu}_{\frac{-1}{2}+i\eta}(u)P^{-\mu}_{\frac{-1}{2}+i\eta}(v) d\eta
\\
&=& \int\limits_1^\infty K(u,v,w)A_w^n\varphi(w)dw.\label{3.2}
\end{eqnarray}
Now, from \cite[p. 165]{AP}, we have
\begin{eqnarray}
|K(u,v,w)|\leq |C_\mu|[(u^2-1)(v^2-1)(w^2-1)]^{-\frac{1}{4}}.\label{3.3}
\end{eqnarray}
In view of \eqref{3.1} and \eqref{3.3}, we can write
\begin{eqnarray*}
\sup\limits_{u\in \mathbb{I}}|e^{-lu}(u^2-1)^rA_u^n(_\mu\mathfrak{T}_u\varphi)(v)|&\leq& |C_\mu| \lambda_{0,m,n}(\varphi) \sup\limits_{u\in \mathbb{I}}|e^{-lu}(u^2-1)^{r-\frac{1}{4}}|
\\
&\times& (v^2-1)^{-\frac{1}{4}} \int\limits_1^\infty (w^2-1)^{-m-\frac{1}{4}}dw.
\end{eqnarray*}
For $r\geq \frac{1}{4},~l> 0$ given supremum at the right hand side is bounded. The involved integral is convergent for $\frac{1}{4}<m<\frac{3}{4}$. Thus for some fixed $v\geq \delta>1$ and some positive constant C, we have
\begin{eqnarray*}
\lambda_{l,r,n}(_\mu\mathfrak{T}_u\varphi)|&\leq&C \lambda_{0,m,n}(\varphi) <\infty.
\end{eqnarray*}
Thus the required result is obtained.
\end{proof}
\begin{theorem}
If the convolution operator  $\varphi\underset{\mu}{*} \psi$ is defined as \eqref{1.13} and $\varphi\in \mathfrak{H}_{0,m,p},~\psi\in \mathfrak{H}_{0,m,q}$ such that $ \frac{1}{4}<m<\frac{3}{4},~p+q=n$, then $\varphi\underset{\mu}{*} \psi \in \mathfrak{H}_{l,r,n}$, where $l>0,~r\geq\frac{1}{4}$ and $ n\in \mathbb{N}_0$.
\end{theorem} 
 \begin{proof}
 Using \eqref{1.13}, we have
 \begin{eqnarray*}
 A_u^n(\varphi\underset{\mu}{*} \psi(u))=\int\limits_1^\infty\int\limits_1^\infty A_u^nK(u,v,w)\varphi(w)\psi(v)dvdw.
 \end{eqnarray*}
 In view of \eqref{1.7} and \eqref{1.10}, we can write
 \begin{eqnarray*}
 A_u^n(\varphi\underset{\mu}{*} \psi(u))&=& \int\limits_1^\infty\int\limits_1^\infty \Bigg(\frac{1}{\pi}(-1)^n \int\limits_0^\infty \Big[\Big(\mu+\frac{1}{2} \Big)^2+\eta^2 \Big]^n \eta \sinh(\pi\eta)\Phi(\eta,\mu) 
\\\nonumber
&& \times P^{-\mu}_{\frac{-1}{2}+i\eta}(u)P^{-\mu}_{\frac{-1}{2}+i\eta}(v)P^{-\mu}_{\frac{-1}{2}+i\eta}(w)d\eta\Bigg)\varphi(w)\psi(v)dvdw
\\\nonumber
&=& \int\limits_1^\infty  \int\limits_1^\infty \Bigg(\frac{1}{\pi} \int\limits_0^\infty  \eta \sinh(\pi\eta)\Phi(\eta,\mu) P^{-\mu}_{\frac{-1}{2}+i\eta}(u)A_v^qP^{-\mu}_{\frac{-1}{2}+i\eta}(v)
\\
&& \times \big(A_w^p P^{-\mu}_{\frac{-1}{2}+i\eta}(w)\big) d\eta\Bigg)\varphi(w)\psi(v)dvdw.
\end{eqnarray*}
 Thus, in view of the self-adjoint property of the differential operator $A_w$, we get
\begin{eqnarray}
\nonumber
&&A_u^n(\varphi\underset{\mu}{*} \psi(u))
\\\nonumber
&&= \frac{1}{\pi}\int\limits_0^\infty P^{-\mu}_{\frac{-1}{2}+i\eta}(u) \Bigg(\int\limits_1^\infty \big(A_w^p P^{-\mu}_{\frac{-1}{2}+i\eta}(w)\big) \varphi(w)dw\Bigg)
\\\nonumber
&\quad&\quad \times \Bigg(\int\limits_1^\infty \big(A_v^q P^{-\mu}_{\frac{-1}{2}+i\eta}(v)\big) \psi(v)dv\Bigg)  \eta \sinh(\pi\eta)\Phi(\eta,\mu)  d\eta
\\\nonumber
&&=\int\limits_1^\infty\int\limits_1^\infty \Bigg(\frac{1}{\pi}\int\limits_0^\infty P^{-\mu}_{\frac{-1}{2}+i\eta}(u)P^{-\mu}_{\frac{-1}{2}+i\eta}(v) P^{-\mu}_{\frac{-1}{2}+i\eta}(w)\eta \sinh(\pi\eta)\Phi(\eta,\mu)  d\eta\Bigg) 
\\\nonumber
&\quad&\quad \times \big(A_w^p\varphi(w)\big)\big(A_v^q\psi(v)\big)dvdw 
\\
&&= \int\limits_1^\infty\int\limits_1^\infty K(u,v,w)\big(A_w^p\varphi(w)\big)\big(A_v^q\psi(v)\big)dvdw .\label{3.4}
\end{eqnarray}
In view of \eqref{3.1} and \eqref{3.3}, we may write
\begin{eqnarray*}
\sup\limits_{u\in \mathbb{I}}|e^{-lu}(u^2-1)^r A_u^n(\varphi\underset{\mu}{*} \psi(u))|&\leq& |C_\mu| \lambda_{0,m,p}(\varphi)\lambda_{0,m,q}(\psi) \sup\limits_{u\in \mathbb{I}}|e^{-lu}(u^2-1)^{r-\frac{1}{4}}|
\\
&\times& \int\limits_1^\infty (v^2-1)^{-m-\frac{1}{4}}dv\int\limits_1^\infty (w^2-1)^{-m-\frac{1}{4}}dw.
\end{eqnarray*}
For $r\geq \frac{1}{4},~l> 0$ given supremum at the right hand side is bounded. The involved integrals are convergent for $\frac{1}{4}<m<\frac{3}{4}$. Thus, we have
\begin{eqnarray*}
\lambda_{l,r,n}(\varphi\underset{\mu}{*} \psi)&\leq&C~ \lambda_{0,m,p}(\varphi)~\lambda_{0,m,q}(\psi) <\infty.
\end{eqnarray*}
Hence, the required result is obtained.

 \end{proof}
The continuity of the $\mu$MFT over function space analogous to $\mathfrak{H}_{l,r,n}$ was discussed earlier in \cite[p. 170]{AP}. Here, we discuss the continuity of the adjoint $\mu$MFT on the function space defined above. 

The adjoint $\mu$MFT is denoted by $\mathfrak{MF^*_\mu}$, is defined for an integrable function $\varphi$ by 
\begin{eqnarray}
\big(\mathfrak{MF^*_\mu}\varphi\big)(u)=\int_0^\infty \varphi(\eta) P_{-\frac{1}{2}+i\eta}^{-\mu}(u)  d\eta. \label{3.5}
\end{eqnarray} 
 \begin{theorem}
 If the adjoint $\mu$MFT is given as \eqref{3.5}, then it is continuous linear mapping from  $\mathfrak{H}_{0,m,0}$ into  $\mathfrak{H}_{l,r,n}$. 
 \end{theorem}
 \begin{proof}
 In view of \eqref{3.1}, \eqref{3.5}, \eqref{1.7}, and \eqref{1.3}, we obtain
 \begin{eqnarray*}
 \lambda_{l,r,n}\big(\mathfrak{MF^*_\mu}\varphi\big)
  &\leq& \sup\limits_{u\in \mathbb{I}}|e^{-lu}(u^2-1)^r P^{-\mathfrak{R}\mu}_{-\frac{1}{2}}(u)|~ \lambda_{0,m,0}(\varphi)\\
 & \quad& \times \int_0^\infty  \left|\Big[\Big(\mu+\frac{1}{2} \Big)^2+\eta^2 \Big]^n\right| (\eta^2-1)^{-m} d\eta.
 \end{eqnarray*}
By \eqref{1.4} and \eqref{1.5}, the supremum in the right hand side of the inequality is finite for $l>0$, $r\geq 0$ and $\mu\geq0$. Moreover, the integral involved converges for $n-m>\frac{1}{2}$. Therefore
 \begin{eqnarray*}
 \lambda_{l,r,n}\big(\mathfrak{MF^*_\mu}\varphi\big) \leq C~  \lambda_{0,m,m}(\varphi)<\infty.
\end{eqnarray*}  
Hence, the result follows. 
 \end{proof}
 \section{Continuity over Sobolev type space}
In \cite{rsp}, the Sobolev space $H_{m,p}$ was studied in terms of the differential operator $D^\alpha$. The sobolev type spaces $\mathbb{S}^{N}_{2}$ and $\mathbb{S}^{N,\alpha}_{p}$ associated with the differential operator $A_x$ corresponding to the Kontorovich-Lebedev transform, were introduced and studied in \cite{sbysobo1}, and \cite{sbysobo2} respectively. For further discussion of such spaces and related works, we refer the reader to \cite{sobolev1,sobolev2,sobolev3,sobolev4,sobolev5,sobolev6}.  Inspired by these studies, we introduce a Sobolev-type space defined in terms of the differential operator $A_u$ given in \eqref{1.6}.
\begin{definition}(\textbf{Sobolev type space $\mathbb{S}^{N,\alpha}_{r,p}(\mathbb{I})$}) 
We define the Sobolev-type space $\mathbb{S}^{N,\alpha}_{r,p}(\mathbb{I})$ as the collection of all functions $\varphi$ such that $A^n_u\varphi(u)\in L^p(\mathbb{I};~e^{-\alpha u}(u^2-1)^rdu),$ where $~\alpha\in \mathbb{R}^0_+=[0,\infty), ~n=0, 1, ...,N, ~$ $r\geq 0, 1\leq p<\infty$. The space $\mathbb{S}^{N,\alpha}_{r,p}(\mathbb{I})$ is equipped with the norm 
\begin{eqnarray}
\|\varphi\|_{N,\alpha,r,p}&=& \Bigg(\sum\limits_{n=0}^N\int\limits_1^\infty |A_u^n\varphi(u)|^pe^{-\alpha u}(u^2-1)^rdu\Bigg)^\frac{1}{p}<\infty,\label{4.1}
\end{eqnarray}
and  the space $\mathbb{S}^{N,\alpha}_{r,\infty}(\mathbb{I})$ is equipped with the norm 
\begin{eqnarray}
\|\varphi\|_{N,\alpha,r,\infty}&=& \sum\limits_{n=0}^N \|A_{(\cdot)}^n\varphi\|_{L_\infty(\mathbb{I};~e^{-\alpha u}(u^2-1)^r du)}<\infty.\label{4.2}
\end{eqnarray}
\end{definition}  
Now, we study various continuity properties in this Sobolev space.
\begin{theorem}
If we consider the function $\varphi\in L^1(\mathbb{I}; (w^2-1)^{\frac{-1}{4}}dw)$, then the translation operation $_\mu\mathfrak{T}_{(\cdot)}\varphi$ continuously maps into $ \mathbb{S}^{N,\alpha}_{r,p}(\mathbb{I}) $, where $r\geq \frac{p}{4}$ and $1\leq p<\infty$.
\end{theorem}
\begin{proof}
Using \eqref{1.12}, \eqref{1.7} and \eqref{1.10}, we get
\begin{eqnarray*}
A_u^n(_\mu\mathfrak{T}_u\varphi(v))=(-1)^n\Bigg[\Big(\mu+\frac{1}{2}\Big)^2+\eta^2\Bigg]^n\int\limits_1^\infty K(u,v,w)\varphi(w)dw.
\end{eqnarray*}
Using \eqref{3.3}, we can write
\begin{eqnarray*}
|A_u^n(_\mu\mathfrak{T}_u\varphi(v))|^p
&\leq& |C_\mu|^p \Bigg[\Big(\mu+\frac{1}{2}\Big)^2+\eta^2\Bigg]^{pn}(u^2-1)^{\frac{-p}{4}}(v^2-1)^{\frac{-p}{4}} 
\\
&\quad& \times\|\varphi\|^p_{L^1(\mathbb{I};~(w^2-1)^{\frac{-1}{4}}dw)}.
\end{eqnarray*}
In view of \eqref{4.1}, we have
\begin{eqnarray*}
\|_\mu\mathfrak{T}_u\varphi\|_{N,\alpha,r,p} &\leq& |C_\mu| \sum\limits_{n=0}^{N}\Bigg[\Big(\mu+\frac{1}{2}\Big)^2+\eta^2\Bigg]^{n}\|\varphi\|_{L^1(\mathbb{I};~(w^2-1)^{\frac{-1}{4}}dw)}
\\
&\quad& \times (v^2-1)^{\frac{-1}{4}}\Big(\int\limits_1^\infty e^{-\alpha u}(u^2-1)^{r-\frac{p}{4}}du\Big)^{\frac{1}{p}}.
\end{eqnarray*}
Now, involved integral is convergent for $\alpha >0,~r\geq \frac{p}{4}, ~1\leq p<\infty$ and also for some fixed $\eta>0, ~v>1$ and constant $C>0$, we can write
\begin{eqnarray*}
\|_\mu\mathfrak{T}_u\varphi\|_{N,\alpha,r,p} &\leq& C ~ \|\varphi\|_{L^1(\mathbb{I};~(w^2-1)^{\frac{-1}{4}}dw)}<\infty.
\end{eqnarray*}
Hence, the translation operator maps into $\mathbb{S}^{N,\alpha}_{r,p}(\mathbb{I})$.
\end{proof}
\begin{theorem}
If we consider the function $\varphi\in L^1(\mathbb{I};~(w^2-1)^{\frac{-1}{4}}dw)$, then the translation operator $_\mu\mathfrak{T}_{(\cdot)}\varphi$ continuously maps into $ \mathbb{S}^{N,\alpha}_{r,\infty}(\mathbb{I}) $, where $r\geq \frac{1}{4}$.
\end{theorem}
\begin{proof}
By proceeding as above result can be obtained.
\end{proof}
\begin{theorem}
If we consider the functions $\varphi\in L^1(\mathbb{I};~(w^2-1)^{\frac{-1}{4}}dw)$ and $\psi\in L^1(\mathbb{I};~(v^2-1)^{\frac{-1}{4}}dv)$, then $\varphi\underset{\mu}{*}\psi\in \mathbb{S}^{N,\alpha}_{r,p}(\mathbb{I}) $, where $~r\geq \frac{p}{4}, ~1\leq p<\infty$.
\end{theorem}
\begin{proof}
Using \eqref{1.13}, \eqref{1.7} and \eqref{1.10}, we get
\begin{eqnarray*}
A_u^n(\varphi\underset{\mu}{*}\psi(u))=(-1)^n\Bigg[\Big(\mu+\frac{1}{2}\Big)^2+\eta^2\Bigg]^n\int\limits_1^\infty K(u,v,w)\varphi(w)\psi(v)dv dw.
\end{eqnarray*}
In view of \eqref{4.1} and \eqref{3.3}, we have
\begin{eqnarray*}
\|\varphi\underset{\mu}{*}\psi\|_{N,\alpha,r,p}&\leq& ~|C_\mu|\sum\limits_{n=0}^N\left|\Bigg[\Big(\mu+\frac{1}{2}\Big)^2+\eta^2\Bigg]^n\right| 
\|\varphi\|_{L^1(\mathbb{I};~(w^2-1)^{\frac{-1}{4}}dw)}
\\
&\times&\|\psi\|_{L^1(\mathbb{I};~(v^2-1)^{\frac{-1}{4}}dv)}
\int\limits_1^\infty (u^2-1)^{r-\frac{p}{4}}e^{-\alpha u}du.
\end{eqnarray*}
Involved integral is convergent for $\alpha>0, ~r\geq \frac{p}{4}, ~1\leq p<\infty$ and also for some fixed $\eta>0$ and constant $C>0$, we can write
\begin{eqnarray*}
\|\varphi\underset{\mu}{*}\psi\|_{N,\alpha,r,p}&\leq& ~C~\|\varphi\|_{L^1(\mathbb{I};~(w^2-1)^{\frac{-1}{4}}dw)}
\|\psi\|_{L^1(\mathbb{I};~(v^2-1)^{\frac{-1}{4}}dv)}<\infty.
\end{eqnarray*}
Hence, the convolution operator maps into $\mathbb{S}^{N,\alpha}_{r,p}(\mathbb{I})$.
\end{proof}
\begin{theorem}
If we consider the functions $\varphi\in L^1(\mathbb{I};(w^2-1)^{\frac{-1}{4}}dw)$ and $\psi\in L^1(\mathbb{I};~(v^2-1)^{\frac{-1}{4}}dv)$, then $\varphi\underset{\mu}{*}\psi\in \mathbb{S}^{N,\alpha}_{r,\infty}(\mathbb{I}) $, where $~r\geq \frac{1}{4}$.
\end{theorem}
\begin{proof}
By proceeding as above required result can be obtained.
\end{proof}
\begin{theorem}
The translation operator defined as \eqref{1.12}, is continuous linear mapping from the Sobolev space $\mathbb{S}^{N,0}_{s,\infty}(\mathbb{I}) ,~0\leq s<\frac{1}{4}$ into $\mathbb{S}^{N,\alpha}_{r,p}(\mathbb{I})$, where $\alpha>0, ~r\geq 0$, $\mu>-\frac{1}{2}$ and $1\leq p<\infty$.
\end{theorem}
\begin{proof}
Using \eqref{3.2}, we have
\begin{eqnarray}
A_u^n(_\mu\mathfrak{T}_u\varphi(v))=
\int\limits_1^\infty K(u,v,w)A_w^n\varphi(w)dw. \label{4.3}
\end{eqnarray}
From \cite[p. 41]{AP2}, considering $\mu$ to be purely real  then for  $\mu>\frac{-1}{2}$ , we have
\begin{eqnarray}
\int\limits_1^\infty K(u,v,w)(w^2-1)^{\frac{\mu}{2}}dw=\frac{(-1)^{\frac{\mu}{2}}}{2^{\mu} \Gamma(1+\mu)}(u^2-1)^{\frac{\mu}{2}}(v^2-1)^{\frac{\mu}{2}}.\label{4.4}
\end{eqnarray}
In view of \eqref{4.1}, \eqref{4.3} and \eqref{4.4}, for some constant $C^\prime>0$, we can write
\begin{eqnarray*}
\|_\mu\mathfrak{T}_{(\cdot)}\varphi\|_{N,\alpha,r,p} \leq ~C^\prime~\|\varphi\|_{N,0,\frac{-\mu}{2},\infty}~(v^2-1)^{\frac{\mu }{2}} \Bigg(\int\limits_1^\infty e^{-\alpha u}(u^2-1)^{r+\frac{\mu p}{2}}du\Bigg)^{\frac{1}{p}}.
\end{eqnarray*}
The involved integral in the RHS is convergent for $\alpha>0,~\mu>\frac{-1}{2}, ~r\geq 0,$ and $ 1\leq p<\infty$. Thus, for some fixed $v>1$, some constant $C>0$ and $0\leq s(=\frac{-\mu}{2})<\frac{1}{4}$, we get
\begin{eqnarray*}
\| _\mu\mathfrak{T}_{(\cdot)}\varphi\|_{N,\alpha,r,p} \leq ~C~\|\varphi\|_{N,0,s,\infty} .
\end{eqnarray*}
Hence, the theorem is proved.
\end{proof}
\begin{theorem}
If $\varphi \in \mathbb{S}^{N,0}_{s,\infty}(\mathbb{I}), ~0\leq s<\frac{1}{4} $ and $\psi\in\mathbb{S}^{0,0}_{\frac{\mu}{2},1}(\mathbb{I}), ~\mu>\frac{-1}{2}$, then the convolution operator defined as \eqref{1.13}, belongs to the Sobolev space $\mathbb{S}^{N,\alpha}_{r,p}(\mathbb{I})$, where $\alpha>0, ~r\geq 0$ and $1\leq p<\infty$.
\end{theorem}
\begin{proof}
 In view of \eqref{1.13}, \eqref{1.7}, \eqref{1.10} and the self adjoint property of the involved differential operator, we can write
 \begin{eqnarray*}
 A_u^n(\varphi\underset{\mu}{*} \psi(u))
&=&\int\limits_1^\infty\int\limits_1^\infty A_w^nK(u,v,w)\varphi(w)\psi(v)dvdw
\\
&=& \int\limits_1^\infty\int\limits_1^\infty K(u,v,w)\psi(v)A_w^n\varphi(w)dvdw.
\end{eqnarray*}
Invoking \eqref{4.1} and \eqref{4.4}, we get
\begin{eqnarray*}
&&\|\varphi\underset{\mu}{*} \psi\|_{N,\alpha,r,p}
\\
&=& \|\varphi\|_{N,0,\frac{-\mu}{2},\infty} \Bigg(\int\limits_1^\infty\Bigg[ \int\limits_1^\infty\Bigg[\int\limits_1^\infty K(u,v,w)(w^2-1)^{\frac{\mu}{2}}dw\Bigg]|\psi(v)|dv\Bigg]^p 
\\
&\times& e^{-\alpha u}(u^2-1)^{r}du\Bigg)^{\frac{1}{p}}
\\
&\leq& C^\prime~\|\varphi\|_{N,0,\frac{-\mu}{2},\infty} \Bigg(\int\limits_1^\infty\Bigg[ \int\limits_1^\infty (v^2-1)^{\frac{\mu}{2}}(u^2-1)^{\frac{\mu}{2}}|\psi(v)|dv\Bigg]^p 
 e^{-\alpha u}(u^2-1)^{r}du\Bigg)^{\frac{1}{p}}
 \\
&\leq& C^\prime~\|\varphi\|_{N,0,\frac{-\mu}{2},\infty}\|\psi\|_{0,0,\frac{\mu}{2},1} \Bigg(\int\limits_1^\infty (u^2-1)^{r+\frac{p\mu}{2}} 
 e^{-\alpha u}du\Bigg)^{\frac{1}{p}},
\end{eqnarray*}
where $C^\prime$ is a positive constant. The involved integral in the right-hand side is convergent for $\alpha>0,~\mu>\frac{-1}{2}, ~r\geq 0,$ and $ 1\leq p<\infty$. Thus, for some constant $C>0$ and $0\leq s(=\frac{-\mu}{2})<\frac{1}{4}$, we get
\begin{eqnarray*}
\|\varphi\underset{\mu}{*} \psi\|_{N,\alpha,r,p}&\leq& C^\prime~\|\varphi\|_{N,0,\frac{-\mu}{2},\infty}\|\psi\|_{0,0,\frac{\mu}{2},1} <\infty.
\end{eqnarray*}
\end{proof}
\textbf{Authors' contributions:}\\ 
U. K. Mandal: Methodology, Validation, Investigation, Conceptualization. \\
Akhilesh Prasad: Supervision, Methodology, Visualization, Review and editing. \\
Sudhanshu Ranjan: Writing-original draft, Formal analysis, Software.\\
\textbf { Declarations:}\\
\textbf{Competing interests:} The authors declare that they have no competing interests.\\
\textbf {Availability of data and materials:} The manuscript has no associated data.\\
\textbf{Acknowledgment:} The second and third authors of the work were supported by the National Board of Higher Mathematics, Gov. of India under sanction no. 02011/7/2021 NBHM (R.P)/RD-II/8164.\\ 
{\bf ORCID:}\\
U. K. Mandal. http://orcid.org/0000-0001-5917-680X\\
Akhilesh Prasad. https://orcid.org/0000-0003-4715-5774\\
\bibliographystyle{amsplain}

\end{document}